\documentclass[12pt,reqno]{amsart}

\usepackage{amsmath, amsthm, amsopn, amssymb, microtype, bm}
\usepackage{mathrsfs}
\usepackage{mathscinet}
\usepackage{bbm}
\usepackage{enumerate, tikz, etoolbox, intcalc, geometry}
\usepackage{caption, subcaption, afterpage}
\usepackage[english]{babel}
\usepackage{enumitem}
\usepackage{tcolorbox}
\usepackage[normalem]{ulem}
\usepackage{aliascnt}
\usepackage[
    colorlinks=true,
    urlcolor=blue,
    pdfborder={0 0 0}
]{hyperref}
\hypersetup{
    linkcolor=[rgb]{0,0,0.6},
    citecolor=[rgb]{0,0.6,0}
}
\usepackage[nameinlink]{cleveref}

\newcommand{\N}{\mathbb{N}}
\newcommand{\Z}{\mathbb{Z}}

\newcommand{\B}{\mathcal{B}}

\newcommand{\PP}{\mathcal{P}}

\graphicspath{{./Pictures/}}

\newcommand{\ignore}[1]{}

\newtheorem{thm}{Theorem}[section]

\newaliascnt{lemma}{thm}
\newtheorem{lemma}[lemma]{Lemma}
\aliascntresetthe{lemma}
\newaliascnt{prop}{thm}
\newtheorem{prop}[prop]{Proposition}
\aliascntresetthe{prop}
\newaliascnt{cor}{thm}
\newtheorem{cor}[cor]{Corollary}
\aliascntresetthe{cor}
\newaliascnt{claim}{thm}

\aliascntresetthe{claim}
\newaliascnt{quest}{thm}

\aliascntresetthe{quest}
\newaliascnt{fact}{thm}
\newtheorem{fact}[fact]{Fact}
\aliascntresetthe{fact}

\theoremstyle{definition}
\newaliascnt{definition}{thm}

\aliascntresetthe{definition}
\newaliascnt{example}{thm}

\aliascntresetthe{example}
\newaliascnt{remark}{thm}

\aliascntresetthe{remark}

\crefname{thm}{Theorem}{Theorems}
\Crefname{thm}{Theorem}{Theorems}

\crefname{lemma}{Lemma}{Lemmas}
\Crefname{lemma}{Lemma}{Lemmas}

\crefname{prop}{Proposition}{Propositions}
\Crefname{prop}{Proposition}{Propositions}

\crefname{cor}{Corollary}{Corollaries}
\Crefname{cor}{Corollary}{Corollaries}

\crefname{claim}{Claim}{Claims}
\Crefname{claim}{Claim}{Claims}

\crefname{quest}{Question}{Questions}
\Crefname{quest}{Question}{Questions}

\crefname{fact}{Fact}{Facts}
\Crefname{fact}{Fact}{Facts}

\crefname{definition}{Definition}{Definitions}
\Crefname{definition}{Definition}{Definitions}

\crefname{example}{Example}{Examples}
\Crefname{example}{Example}{Examples}

\crefname{remark}{Remark}{Remarks}
\Crefname{remark}{Remark}{Remarks}

\crefname{conj}{Conjecture}{Conjectures}
\Crefname{conj}{Conjecture}{Conjectures}

\crefname{section}{Section}{Sections}
\Crefname{section}{Section}{Sections}

\crefname{figure}{Figure}{Figures}
\Crefname{figure}{Figure}{Figures}

\title{Alpern Lemma Independent of a Partition for Aperiodic Systems}

\author{Nishant Chandgotia}
\address{Tata Institute of Fundamental Research - Center for Applicable Mathematics}
\email{nishant.chandgotia@gmail.com}

\author{Zemer Kosloff}
\address{School of Mathematics, University of Bristol, Fry Building,
Woodland Road, Bristol BS8 1UG, United Kingdom\newline
Einstein Institute of Mathematics, The Hebrew University of Jerusalem,
Edmond J. Safra Campus, Givat Ram, Jerusalem 9190401, Israel}
\email{zemer.kosloff@bristol.ac.uk}
\email{zemer.kosloff@mail.huji.ac.il}

\author{Shrey Sanadhya}
\address{School of Mathematics and Statistics, University of Sydney, NSW 2006, Australia}
\email{shrey.sanadhya@sydney.edu.au}


\begin{document}

\maketitle

\begin{abstract}
We prove that for every $(X,\B,\mu,T)$, an aperiodic, probability preserving system and $\mathcal{P}$ a finite measurable partition of $X$, there exists an Alpern Multitower whose base is independent of $\mathcal{P}$. This removes the ergodicity assumption in the original proof of the independent Alpern Lemma of \cite{MR3373823}. As a consequence, we remove ergodicity assumption from recent results on limit theorems in deterministic systems \cite{Kosloff_Volny_2022,kosloffVolnyStableCLT,MR4955410}.
\end{abstract}
\section{Introduction}

The Rokhlin Lemma says that  for every aperiodic dynamical system $(X,\B,\mu,T)$, and a given integer $N$ and $\epsilon>0$, one can find a set $B$ whose translates under the transformation up to time $N$ are pairwise disjoint and they cover the whole space up to $\mu$ measure $\epsilon$. The sequence $B,TB,\ldots, T^NB$ is called a \textbf{Rokhlin tower of height N and error $\epsilon$} and $B$ is \textbf{the base of the tower}. It is a corner stone in the structure theory of dynamical systems and has far reaching consequences such as determining the largeness in the Baire category sense of subsets of measure preserving transformations, Ornstein's proof of the isomorphism theorem for Bernoulli shifts \cite{MR447525}, behavior of cocycles \cite{Lind_1978,Kornfeld_1999,Kornfeld_Losert_2004} and much more, for a review see for example \cite{MR2087594}. 

The Rokhlin Lemma Independent of a Partition, which was introduced by Katznelson and Weiss \cite{Katznelson_Weiss_1972}, states that given a finite measure partition $\PP$ of the space, one can ensure, in addition, that  the base of the tower $B$ is independent of the partition. The first reason to introduce this lemma was to remove the weak mixing assumption from Ornstein's work and this enhancement has found many applications since; among the notable ones is in Volny's proof that every aperiodic dynamical system has a cocycle which satisfies the functional central limit theorem \cite{MR1624218}.

The Alpern Lemma says that for any positive integer $N$ one can find  $B_N,B_{N+1}\subset X$, two bases of $N$ and $N+1$ Rokhlin towers respectively such that the union of the two Rokhlin towers has full measure. In what follows, the towers arising from the Alpern Lemma will be referred to as Alpern towers. See \cite{MR2431052} for a discussion about this lemma and its consequences and \cite{MR2901368} for a more recent application. 

In \cite{MR3373823} it was proved that for any $(X,\B,\mu,T)$ an aperiodic, ergodic, probability preserving system, $\mathcal{P}$ a finite measurable partition of $X$ and integer $N$, there exist $N,N+1$  Alpern towers whose bases $B_N$ and $B_{N+1}$ are independent of the partition. This result has been applied by Kosloff and Volny in \cite[Proposition 1]{Kosloff_Volny_2022} to prove that any finite valued triangular array can be embedded in an aperiodic, ergodic system. This embedding is the first pillar in the recent advances on probabilistic limit theorems in deterministic ($0$ entropy) dynamical systems. 

The goal of this note is to remove the ergodicity assumption from the independent Alpern Lemma. 
\begin{thm}\label{Thm: Aperiodic} Let $(X,\mu, \B, T)$ be an aperiodic probability preserving $\Z$ action. For any $N \in \N$ and a finite measurable partition $\PP = \{P_1,P_2,\ldots,P_t\}$ of $X$, there exist $B_N, B_{N+1} \in \B$ such that
\[
\{T^j B_h :  0\leq j < h, h\in \{N,N+1\}\}
\]
is a partition of $X$ up-to measure zero and $B_N, B_{N+1}$ are independent of $\PP$. 
\end{thm}
In \Cref{sec: Alpern proof}, we first explain the terminology and then prove \cref{Thm: Aperiodic} building upon the very clever proof of the ergodic case in \cite{MR3373823}. In \Cref{sec: application to limit theorems}, we explain the consequence of this result regarding recent limit theorem in dynamical systems. 

{\bf Acknowledgement.} N.C. thanks the Department of Atomic Energy, Government of India, under
project identification no. RTI 4014. Z. K. was supported by the ISF (grant No. 1180/22). S.S. was supported by
the Australian Research Council grant DP240100472.   
 
\section{Proof of \cref{Thm: Aperiodic}}\label{sec: Alpern proof}
A probability preserving system is a measurable, almost everywhere invertible map $T:X\to X$ of a probability space $(X,\B,\mu)$. It is \textbf{aperiodic} if the collection of periodic points for $T$ is a $\mu$ null set. That is
\[
\mu\left(\left\{x\in X:\ \exists n\in\N, T^nx=x\right\}\right)=0.
\]
A collection $\mathcal{P}=(a)_{a\in \mathcal{P}}$ of subsets of $X$ is a \textbf{finite measurable partition} of $X$ if the cardinality of $\mathcal{P}$ is finite, all atoms are $\B$ measurable and $\mu\left(X\setminus \uplus_{a\in F}P_a\right)=0$. Finally, a subset $B\in \B$ is \text{independent} of a finite measurable partition $\PP$ if for all $a\in \PP$, 
\[
\mu(B\cap a)=\mu(B)\mu(a). 
\]
As stated in the introduction, \cref{Thm: Aperiodic} was proved for ergodic, aperiodic, probability preserving systems in \cite{MR3373823}. In \cite{MR3941195} the authors further improved the result so that the mass of the $N,N+1$ towers is arbitrary. The extension of this result to aperiodic (not necessarily ergodic) transformations is carried out by following closely the mechanism of the proof in \cite{MR3373823}. A key point is carried out in \Cref{lem: Aperiodic partittion} which is to partition the space to a finite collection of $T$ invariant sets  according to the behavior in the ergodic theorem of the sums of the vector valued function $\phi(x)=\left(1_p(x)\right)_{p\in\mathcal{P}}$.
For this reason we first start with a recap on the results and proof in  \cite{MR3373823, MR3941195}.
\subsection{Recap on the proof of Alpern Lemma independent of a partition for ergodic actions}
Given a probability preserving transformation $(X,\mathcal{B},\mu,T)$ and $A\in \mathcal{B}$, it follows from the Poincar\'e Recurrence Theorem that the function $n:X\to \N\cup\{\infty\}$ defined by 
\[
n(x):=\inf\left\{n\in\N:\ T^nx\in A\right\},
\]
is finite for $\mu$ almost every $x\in A$. In the case when $x\in A$ and $n(x)=n$, one says that \textbf{the return time of $x$ to $A$ is $n$}.

The first step in the proof of \cite[Theorem 1.2]{MR3373823} is the following.

\begin{lemma} \cite[Lemma 2.2]{MR3373823}\label{lem: ergodic theorem Kalikow}
Let $(X,\mathcal{B},\mu,T)$ be an ergodic probability preserving transformation. Let $M\in\N$ and let $\mathcal{P}=\left\{P_1,\ldots,P_t\right\}$, a finite partition such that  $\mu\left(P_i\right)>0$ for all $1\leq i \leq t$. Then there exists $S\in \mathcal{B}$ of positive measure such that if $x\in S$ with return time $n(x)=n$, then for all $1\leq i\leq t$,
\[
\left|\left\{x,Tx,\ldots,T^{n-1}x\right\}\cap P_i\right|>M. 
\]
\end{lemma}
A subset $S\in \mathcal{B}$ such that $\cup_{n=1}^\infty T^{-n}S=X$ modulo $\mu$ is called an \textbf{exhaustive set for $T$}. Note that when $T$ is ergodic, every set is exhaustive. Given an exhaustive set $S$ for $T$ and $l\in\N$ we write $S_l:=\{x\in S: n(x)=l\}$, where $n(x)$ for $x \in S$ is the return time to $S$. Note that $X$ can be represented as a (Kakutani) tower over $S$ via
\[
X=\bigsqcup_{l=1}^\infty \bigsqcup_{j=0}^{l-1}T^{-j}S_l\ \mathrm{mod}\mu.
\]
The set $S$ is called \textbf{the base of the tower}. 

Using \cref{lem: ergodic theorem Kalikow} the authors in \cite{MR3373823} deduce the following. 
\begin{cor} \cite[Corollary 2.3]{MR3373823}\label{Cor: ergodic Kalikow}
Let $(X,\mathcal{B},\mu,T)$ be an ergodic probability preserving transformation. Let $M\in\N$ and let $\mathcal{P}=\left\{P_1,\ldots,P_t\right\}$, a finite partition such that  $\mu\left(P_i\right)>0$ for all $i$. There exists a Kakutani tower with base $S=\bigsqcup_{j=Mt}^\infty S_j$ such that for every $x\in S_j$, 
\[
\left|\left\{x,Tx,\ldots,T^{j-1}x\right\}\cap P_i\right|>M.
\]

\end{cor}
With this Corollary at hand and a very clever combinatorial argument \cite[Theorem 1]{MR3373823}, which is \cref{Thm: Aperiodic} for ergodic $T$, is deduced. In \cite{MR3941195}, starting with \cite[Theorem 1]{MR3373823}, the following refinement is proved.  

\begin{thm}\cite{MR3941195}\label{Thm: Ergodic McCutcheon et al.} Let $(X,\mu, \B, T)$ be an ergodic, aperiodic probability preserving $\Z$ action. For any $N \in \N$ and a finite measurable partition $\PP = \{P_1,P_2,\ldots,P_t\}$ of $X$. There exist $B_N, B_{N+1} \in \B$ such that

\begin{enumerate}[label=(\alph*)]
\item $\{T^j B_t :  0\leq j < t, t\in \{N,N+1\}\}$ is a partition of $X$ up-to measure zero.
\item $B_N$ and $B_{N+1}$ are independent of $\mathcal{P}$.
\item $\mu\left(\bigsqcup_{j=0}^{N-1}T^j B_N\right)=\mu\left(\bigsqcup_{j=0}^{N}T^j B_{N+1}\right)=\frac{1}{2}$.
\end{enumerate}

\end{thm}

It is important to note that \Cref{Cor: ergodic Kalikow} is the only ergodic theoretic input in the proof of \cite[Theorem 1.2]{MR3373823} and that the deduction of \cref{Thm: Ergodic McCutcheon et al.} from \cite[Theorem 1.2]{MR3373823} is via a direct argument. We deduce from this the following convenient summary or reformulation. For $t\in\N$, we will use the shorthand notation $[t]:=\{1,\ldots,t\}$.

\begin{fact}\label{fact: Kalikow}
Let $(X,\mathcal{B},\mu,T)$ be a probability preserving transformation and $\mathcal{P}=\left\{P_1,\ldots,P_t\right\}$, a finite partition such that  $\mu\left(P_i\right)>0$ for all $i$. If for all $M\in\N$, there exists a Kakutani tower for $T$ with base $S=\bigsqcup_{j=Mt}^\infty S_j$ such that for every $x\in S_j$, 
\[
\min_{i\in [t]}\left|\left\{x,Tx,\ldots,T^{j-1}x\right\}\cap P_i\right|>M.
\]
then there exists $B_N, B_{N+1} \in \B$ such that
\begin{itemize}
\item[(a)] $
\{T^j B_t :  0\leq j < t, t\in \{N,N+1\}\}$ 
is a partition of $X$ up-to measure zero.
\item[(b)] $B_N$ and $B_{N+1}$ are independent of $\mathcal{P}$.
\item[(c)] $\mu\left(\bigsqcup_{j=0}^{N-1}T^j B_N\right)=\mu\left(\bigsqcup_{j=0}^{N}T^j B_{N+1}\right)=\frac{1}{2}$.
\end{itemize}
\end{fact}
\subsection{Proof of \Cref{Thm: Aperiodic}}

In \cite{MR3941195} it is remarked that \Cref{Thm: Aperiodic} can be deduced from the ergodic case by using the ergodic decomposition. While it seems very plausible at first glance, there are several issues which come up. The first is that if $\mathcal{P}$ is a partition of the space $X$, it could be that some of its atoms degenerate on a positive measure set of ergodic components. Another issue is that even if we have only two ergodic components and two Alpern towers independent of $\mathcal{P}$, their union may not be independent of $\mathcal{P}$. The final one, is that as in many cases the space of ergodic components is uncountable, it is not clear how to take an uncountable union of Alpern towers.

The first step in our proof is to divide the space into invariant subsets according to which parts of the partition are present. From now on we consider $(X,\mathcal{B},\mu,T)$ a fixed aperiodic dynamical system and $\mathcal{P}=\{P_1,\dots,P_t\}$ a fixed measurable partition, each of its elements of positive $\mu$ measure. Write $\mathcal{I}\subset\mathcal{B}$ for the $\sigma$ algebra of $T$ invariant  subsets and define $\varphi:X\to [0,1]^t$, a $\mathcal{I}$ measurable function, via
\[
\varphi(x):=\left(\mathbb{E}\left(\left. 1_{P_1}\right|\mathcal{I}\right)(x),\ldots,\mathbb{E}\left(\left. 1_{P_t}\right|\mathcal{I}\right)(x)\right)
\]
For a subset $F\subset [t]$, write 
\[
X_F:=\left\{x\in X:\ \varphi(x)_j\neq 0\ \text{if and only if}\ j\in F\right\}.
\]
Write $\mathscr{K}\subset 2^{[t]}\setminus\emptyset$ for the collection of all $F$'s such that $\mu\left(X_F\right)>0$.
\begin{lemma}\label{lem: Aperiodic partittion} 
$\left\{X_F:\ F\in\mathscr{K}\right\}$ is a partition of $X$ to $T$ invariant sets so that for each $ F\in\mathscr{K}$ the following properties are satisfied: 
\begin{itemize}
    \item For all $j\in [t]$, $\mu\left(X_F\cap P_j\right)>0$ if and only if $j\in F$. Consequently $\mathcal{P}_F:=\left\{X_F\cap P_j:\ j\in F\right\}$ is a partition of $X_F$.
    \item For $\mu$ almost all $x\in X_F$, for all $j\in F$, 
    \[
    \lim_{n\to\infty}\frac{\left|\left\{x,Tx,\ldots,T^{n-1}x\right\}\cap \left(X_F\cap P_j\right)\right|}{n}>0.
    \]
\end{itemize}
\end{lemma}
\begin{proof} For almost every $x\in X$, $F_x:=\{j\in[t]:\varphi(x)_j\neq 0\}$ is a nonempty subset of $[t]$, thus we have $x\in X_{F_x}$. Let $F,G \in\mathscr{K}$ such that $F\neq G$. Then without loss of generality there exists $j \in F$ such that $j \notin G$. Thus $\varphi(x)_j\neq 0$ for a.e $x\in X_F$, while $\varphi(x)_j=0$ for a.e $x\in X_G$, hence $X_F\cap X_G=\emptyset$. This shows that $\{X_F:F\in\mathscr K\}$ is a partition of $X$. The sets $X_F$ are $T$ invariant because $\varphi$ is $\mathcal{I}$ measurable  and $$X_F:=\bigcap_{j\in F}\{\varphi(x)_j\neq 0\}\cap \bigcap_{j\notin F}\{\varphi(x)_j=0\}.$$

Now fix $ F\in \mathscr{K}$ and note that for all $j\in [t]$, 
\begin{align*}
\mu\left(X_F\cap P_j\right)&= \int 1_{X_F}1_{P_j}d\mu,\\
&= \int \mathbb{E}\left(\left.1_{X_F}1_{P_j}\right|\mathcal{I}\right)d\mu=\int_{X_F} \mathbb{E}\left(\left.1_{P_j}\right|\mathcal{I}\right)d\mu, \ \text{since $1_{X_F}$ is $\mathcal{I}$-measurable.}
\end{align*}
The first item follows from this, the positivity of conditional expectation and the definition of $X_F$. The last claim follows from the pointwise ergodic theorem since for every $j\in F$, for $\mu$ almost every $x\in X_F$
\[
\lim_{n\to\infty}\frac{1}{n}\sum_{k=0}^{n-1}1_{P_j}(T^k x)=\mathbb{E}\left(\left. 1_{P_j}\right|\mathcal{I}\right)(x)=\varphi(x)_j>0.
\]\end{proof}
In view of \Cref{lem: Aperiodic partittion}, for each $F\in \mathscr{K}$ write $\B_F=X_F\cap \B$ and $\mu_F$ for the measure $\mu$ conditioned on $X_F$. Since $X_F$ is $T$ invariant, we note that $\left(X_F,\mathcal{B}_F,\mu_F,T\right)$ is an aperiodic dynamical system and $\mathcal{P}_F$ is a partition of $X_F$ to pieces, all of whose pieces have positive measure. We first prove an analogue of \cref{lem: ergodic theorem Kalikow} for these aperiodic subsystems.

\begin{lemma}\label{lem: aperiodic tower 1}
Let $F\in\mathscr{K}$ and $M\in\N$. For every $A\subset X_F$, a $T$ invariant set with $\mu_F(A)>0$, there exists $S\subset A$ with $\mu_F(S)>0$ such that for $\mu_F$ almost all $x\in S$ whose return time to $S$ is $n(x)$, 
\[
\min_{Q\in\mathcal{P}_F}\left|\left\{x,Tx,...,T^{n(x)-1}x\right\}\cap Q\right|>M.
\]
\end{lemma}
The proof of \cref{lem: aperiodic tower 1} is essentially the same as the proof of \cite[Lemma 2.2]{MR3373823} where we use the properties of $X_F$ from \cref{lem: Aperiodic partittion} instead of the ergodic theorem and the choice of the set $S$ in its final part needs a slightly more subtle input which is a part of the proof of the Rokhlin Lemma for aperiodic actions. 
\begin{proof}
Define for $K\in\N$,
\[
A_K:=\left\{x\in A:\ \min_{Q\in\mathcal{P}_F}\sum_{j=0}^{K-1}1_Q\circ T^j(x)> M\right\}
\]
By \Cref{lem: Aperiodic partittion}, $A_K\uparrow A$ (mod $\mu$ null sets) and consequently there exists $K\in \N$ such that the set $A_K$ is of positive $\mu$ measure. Note that for all $x\in A_K$ and every $Q\in\mathcal{P}_F$,
\[
\left|\left\{x,\ldots,T^{K-1}x\right\}\cap Q\right|>M.
\]
 By \cite[Lemma 1]{MR9703} (see also the proof of the Rokhlin lemma in \cite{MR97489}) there exists $S\subset A_K$ of positive $\mu$ measure such that for all $1\leq j\leq K-1$, $S\cap T^{j}S=\emptyset$, in other words for all $x\in S$ the return time to $S$ is larger or equal to $K$. The set $S$ is the desired set. 
\end{proof}
For a set $B\in\B$ we write $[B]_T:=\cup_{n\in\Z}T^nB$ for the $T$ saturation of $B$. 
\begin{cor}\label{cor: Aperiodic tower}
Let $F\in\mathscr{K}$. For all  $M\in\N$, there exists a Kakutani tower for $\left(X_F,\mathcal{B}_F,\mu_F,T\right)$ with base $S=\bigsqcup_{j=Mt}^\infty S_j$ such that for every $x\in S_j$, 
\[
\min_{Q\in\mathcal{P}_F}\left|\left\{x,Tx,\ldots,T^{j-1}x\right\}\cap Q\right|>M.
\]
\end{cor}
\begin{proof}
By \cref{lem: aperiodic tower 1} there exists $C^{(1)}\subset X_F$ of positive measure such that for all $x\in C^{(1)}$ whose return time to $C^{(1)}$ is $n(x)$, 
\[
\min_{Q\in\mathcal{P}_F}\left|\left\{x,Tx,...,T^{n(x)-1}x\right\}\cap Q\right|>M.
\]

If $\left[C^{(1)}\right]_T$ is not of full, consider now $A^{(1)}:=X_F\setminus \left[C^{(1)}\right]_T$ and note that it is a $T$ invariant set. If $A^{(1)}$ is a null set, then $S^{(1)}=C^{(1)}$ is the desired base set. Otherwise applying \cref{lem: aperiodic tower 1}, once again we obtain $C^{(2)}\subset A^{(1)}$ of positive measure such that for all $x\in C^{(2)}$ whose return time to $C^{(2)}$ is $n(x)$, 
\[
\min_{Q\in\mathcal{P}_F}\left|\left\{x,Tx,...,T^{n(x)-1}x\right\}\cap Q\right|>M.
\]
Now write $S^{(2)}:=C^{(1)} \bigsqcup C^{(2)}$ and note that by the way we chose $C^{(1)}$ and $C^{(2)}$ the return time of any $x\in S^{(2)}$ is the same as its return time to $C^{(1)}$ or $C^{(2)}$ according to which set it belongs, so $n(x)$ is its return time to $S^{(2)}$. Again if $\left[S^{(2)}\right]_T$ is of full $\mu_F$ measure, we are done.

Continuing in this way, by Zorn's lemma we obtain a maximal collection of sets $\left(C^{(\alpha)}\right)_{\alpha\in \mathcal{A}}$ such that the following holds:
\begin{itemize}
    
    \item For all $\alpha\in \mathcal{A}$,  $C^{(\alpha)}$ is a measurable subset of $X_F$ and 
     \[
\mu_F\left( C^{(\alpha)}\right)>0.
\]
    
    \item For every $\alpha_1\neq \alpha_2$, the $T$ saturations of $C^{(\alpha_1)}$ and $C^{(\alpha_2)}$ are disjoint. 
    \item For all $\alpha\in \mathcal{A}$, for all $x\in C^{(\alpha)}$ whose return time to $C^{(\alpha)}$ is $n(x)$, 
\[
\min_{Q\in\mathcal{P}_F}\left|\left\{x,Tx,...,T^{n(x)-1}x\right\}\cap Q\right|>M.
\] 
\end{itemize} 
As the $T$ saturations of the $C^{(\alpha)}$'s are of positive $\mu$ measure and pairwise disjoint, we deduce that $\mathcal{A}$ is countable and write simply $\mathcal{A}=\N$ and $(C^{(n)})_n$. 

We claim that $S:=\cup_n C^{(n)}$ is the required base of a tower. Indeed, if $\mu_F\left([S]_T\right)<1$, we can find using \Cref{lem: aperiodic tower 1} a set $C'\in\mathcal{B}$ which satisfies $\mu_F(C')>0$, $C'\cap [S]_T=\emptyset$ and for almost every $x\in C'$ whose return time to $C'$ is $n(x)$ one has  
\[
\min_{Q\in\mathcal{P}_F}\left|\left\{x,Tx,...,T^{n(x)-1}x\right\}\cap Q\right|>M.
\]
As this is a contradiction for the maximality of the family, we deduce that $\mu_F([S]_T)=1$. In addition, for every $x\in S$ there exists a unique $n\in\N$ such that $x\in C^{(n)}$. As the $T$ saturations of the subsets in the maximal family are pairwise disjoint, we deduce that $n(x)$, the return time of $x$ to $S$ is the same as the return time of $x$ to $C^{(n)}$. This and the choice of $C^{(n)}$'s implies that for all $x\in S$ whose return time to $S$ is $n(x)$, one has 
\[
\min_{Q\in\mathcal{P}_F}\left|\left\{x,Tx,...,T^{n(x)-1}x\right\}\cap Q\right|>M,
\]
as desired. 
\end{proof}
\begin{proof}[Proof of \Cref{Thm: Aperiodic}.] 
Let $F\in\mathscr{K}$. By \Cref{cor: Aperiodic tower}, the aperiodic dynamical system $\left(X_F,\mathcal{B}_F,\mu_F,T\right)$ satisfies the conditions of \cref{fact: Kalikow} with respect to its partition $\mathcal{P}_F$. Consequently there exists $B_{N,F},B_{N+1,F}\subset X_F$ satisfying 
\begin{itemize}
\item[(a)] $
\{T^j B_{t,F} :  0\leq j < t, t\in \{N,N+1\}\}$ 
is a partition of $X_F$ up-to measure zero.
\item[(b)] $B_{N,F}$ and $B_{N+1,F}$ are independent of $\mathcal{P}_F$.
\item[(c)] $\mu_F\left(\bigsqcup_{j=0}^{N-1}T^j B_{N,F}\right)=\mu_{F}\left(\bigsqcup_{j=0}^{N}T^j B_{N+1,F}\right)=\frac{1}{2}$. As $T$ preserves $\mu_F$ this implies that for every  $t\in\{N,N+1\}$, $\mu_F\left(B_{t,F}\right)=\frac{1}{2t}$.
\end{itemize}
We claim that $B_N:=\bigsqcup_{F\in\mathscr{K}}B_{N,F}$ and $B_{N+1}:=\bigsqcup_{F\in\mathscr{K}}B_{N+1,F}$ form the bases of an $N,N+1$ Alpern tower independent of $\mathcal{P}$. 

The fact that $\{T^j B_{t} :  0\leq j < t, t\in \{N,N+1\}\}$ are pairwise disjoint follows from the fact that for each $F\in\mathscr{K}$, $\{T^j B_{t,F} :  0\leq j < t, t\in \{N,N+1\}$
are pairwise disjoint subsets of $X_F$ and that if $F_1,F_2\in\mathscr{K}$ are distinct then  $X_{F_1}\cap X_{F_2}=\emptyset$. To see that they form a partition of $X$ (mod $\mu$) note that \footnote{all equalities are modulo null sets}
\begin{align*}
\bigsqcup_{t\in \{N,N+1\}}\bigsqcup_{j=0}^{t-1}T^{j}B_t&=\bigsqcup_{F\in\mathscr{K}}\bigsqcup_{t\in \{N,N+1\}}\bigsqcup_{j=0}^{t-1}T^{j}B_{t,F}\\
&=\bigsqcup_{F\in\mathscr{K}}X_F=X.
\end{align*}
To complete the proof we need to show that both $B_N$ and $B_{N+1}$ are independent of $\PP$. Let $P_i \in \PP$, then
\[
\mu(P_i \cap B_N) = \sum_{F \in \mathscr{K}} \mu(P_i \cap B_N(F)).
\]
For all $F\in \mathscr{K}$ such that $i\notin F$, $\mu\left(P_i\cap X_F\right)=0$. In addition, if $i\in F$, then $P_i\cap X_F\in\mathcal{P}_F$. Combining these two insights we get,
\begin{align*}
\mu(P_i \cap B_N) &= \sum_{F\in\mathscr{K} : i \in F} \mu(X_F) \cdot \dfrac{\mu\left(P_i \cap B_{N,F}\right)}{\mu(X_F)}\\
&= \sum_{F\in\mathscr{K} : i \in F} \mu(X_F) \cdot \mu_F\left(\left(P_i\cap X_F\right) \cap B_{N,F}\right). 
\end{align*} 

Recall that $\PP_F$ is a partition of $X_F$ which is independent of $B_{N,F}$, hence we get
\begin{align}
    \mu(P_i \cap B_N) &= \sum_{F\in\mathscr{K} : i \in F} \mu(X_F) \cdot \mu_F\left(P_i\right) \mu_F\left(B_{N,F}\right)\nonumber\\
    &= \dfrac{1}{2N} \sum_{F\in\mathscr{K} : i \in F} \mu(X_F) \cdot \mu_F(P_i)\nonumber\\
    &=\dfrac{1}{2N} \sum_{F\in\mathscr{K} : i \in F} \mu(X_F\cap P_i)=\frac{\mu\left(P_i\right)}{2N}.\label{eq: 1}
\end{align} 
Here we used the fact that $\mu_F$ measure of the $B_{N,F}$  is $1/2N$.  A similar argument shows that
\begin{align}
\mu\left(B_N\right)&=\sum_{F\in\mathscr{K}}\mu\left(X_F\right)\mu_F\left(B_{N,F}\right)\nonumber\\
&=\frac{1}{2N}\sum_{F\in\mathscr{K}}\mu\left(X_F\right)=\frac{1}{2N}.\label{eq: 2}
\end{align}

We deduce from \eqref{eq: 1} and \eqref{eq: 2}  that for all $1\leq i\leq t$,
\[
\mu(P_i \cap B_N)=\mu\left(P_i\right)\mu\left(B_N\right),
\]
in other words, $B_N$ is independent of $\mathcal{P}$. A similar argument shows that $B_{N+1}$ is independent of $\PP$ as well. 
\end{proof}
\section{Application to limit theorems in deterministic dynamical systems}\label{sec: application to limit theorems}
The ergodic case of \cref{Thm: Aperiodic} was a main tool in the proof of \cite[Proposition 2]{Kosloff_Volny_2022} which is the main backbone of many recent results on limit theorems in deterministic systems \cite{Kosloff_Volny_2022,kosloffVolnyStableCLT,MR4955410}. With \cref{Thm: Aperiodic} we can remove the ergodicity assumption. 

\begin{prop}\label{prop: copying lemma}
Let $(X,\mathcal{B},\mu,T)$ be an aperiodic invertible probability measure preserving transformation and $\xi$ a finite measurable partition of $X$. For every finite set $A$ and $X_1, X_2, .., X_l$ a
collection of $A$ valued i.i.d. random variables, there exists $f : X \to A$ such that $(f\circ T^i)_{i=0}^{l-1}$ and $(X_i)_{i=1}^l$ are equally distributed and  $(f\circ T^i)_{i=0}^{l-1}$ is independent of $\xi$.     
\end{prop}
\begin{proof}
The proof is identical to \cite[Proposition 2]{Kosloff_Volny_2022}, the only change one needs to make is to replace \cite[Corollary 1]{MR3941195} with \cref{Thm: Aperiodic} in the beginning of the proof.
\end{proof}
We now illustrate its utility. Let $(X,\B,\mu,T)$ be an aperiodic (not necessarily ergodic) probability preserving system. For a positive integer $k$, let $d_k=2^{k^2}$, 
\[
p_k=\begin{cases}
2^k, &\ k\ \text{even},\\
2^{k-1}+1,&\ k\ \text{odd},
\end{cases}
\]
and $\alpha_k:=\frac{1}{p_k\sqrt{k\log (k+1)}}$ where $\log$ is the logarithm to base $2$. 

A finite valued function is fully described by a finite partition on which the function is constant on each atom together with the values it receives on each atom. Therefore, an inductive application of \Cref{prop: copying lemma} shows that there exists  a sequence of functions $\bar{f}_k:X\to \{-1,0,1\}$ with the following properties:
\begin{itemize}
\item for every $ k\in\N$, $\left(\bar{f}_k\circ T^j\right)_{j=0}^{2d_k+p_k}$ is distributed an i.i.d. sequence with 
\[
\mu(\bar{f}_k=1)=\mu(\bar{f}_k=-1)=\frac{\alpha_k^2}{2},\ \ \mu(\bar{f}_k=0)=1-\alpha_k^2.
\]
\item for every $2\leq k\in\N$,  $\left(\bar{f}_k\circ T^j\right)_{j=0}^{2d_k+p_k}$ is independent of 
\[
\left\{\bar{f}_l\circ T^j:\ 1\leq l<k,\ 0\leq j\leq 2d_k+p_k\right\}.
\]
\end{itemize}
Define $f=\sum_{k=1}^\infty f_k$ where 
\[
f_k:=\sum_{j=0}^{p_k-1}\left(\bar{f}_k\circ T^j-\bar{f}_k\circ T^{j+d_k}\right).
\]
It is shown in \cite[Proposition 3]{Kosloff_Volny_2022} that $f\in L^2(X,\mu)$ and in \cite[Theorem 4]{Kosloff_Volny_2022} that $f$ satisfies the lattice local central limit theorem with variance $\sigma^2=2(\ln2)^2$. As a consequence we have recovered the following version of the main result of \cite{Kosloff_Volny_2022} without the extra assumption of ergodicity. 

\begin{thm}
For every aperiodic, probability preserving system $(X,\B,\mu,T)$, there exists a square integrable $f:X\to \Z$ with $\int fd\mu=0$ which satisfies the lattice local central limit theorem. That is there exists $\sigma^2>0$ such that 
\[
\sup_{x\in\Z}\left|\sqrt{n}\mu(S_n(f)=x)-\frac{1}{\sqrt{2\pi\sigma^2}}\exp\left(-\frac{x^2}{2n\sigma^2}\right)\right|\xrightarrow[n\to\infty]{}0,
\]
where $S_n(f):=\sum_{k=0}^{n-1}f\circ T^k$. 
\end{thm}

A similar argument can be made to remove the assumption of ergodicity of the transformation in the limit theorems in \cite{Kosloff_Volny_2022,kosloffVolnyStableCLT,MR4955410}.

\bibliographystyle{alpha}
\bibliography{lib}

\end{document}